\documentclass[12pt]{amsart}
\usepackage[utf8]{inputenc}
\usepackage{amsfonts}
\usepackage{amsmath}
\usepackage{amssymb}
\usepackage{amsthm}
\usepackage{xcolor}
\usepackage[margin=1.3in]{geometry}
\usepackage{graphicx}
\usepackage[11pt]{moresize}
\usepackage{tikz}
\tikzstyle{vertex}=[circle,draw=black,fill=black,inner sep=0,minimum size=3pt,text=white,font=\footnotesize]

\newtheorem{theorem}{Theorem}
\newtheorem*{conjecture*}{Conjecture}

\newtheorem{lemma}{Lemma}
\newtheorem{corollary}{Corollary}[section]
\theoremstyle{remark}
\newtheorem{definition}{Definition}[section]

\newtheorem*{remark*}{Remark}

\newcommand{\vs}{\vspace{3mm}}

\newcommand{\Z}{\mathbb{Z}}

\newcommand{\N}{\mathbb{N}}

\title{On the Iterates of Digit Maps}
\author{Zachary Chase}
\address{Mathematical Institute, Andrew Wiles Building, Radcliffe Observatory Quarter, Woodstock Road, Oxford OX2 6GG, UK}
\email{zachary.chase@maths.ox.ac.uk}
\date{May 1, 2020}

\begin{document}

\begin{abstract}
Given a base $b$, a ``digit map" is a map $f: \mathbb{Z}^{\ge 0} \to \mathbb{Z}^{\ge 0}$ of the form $f(\sum_{i=0}^n a_ib^i) = \sum_{i=0}^n f_*(a_i)$, $0 \le a_i \le b-1$ for each $i$, where $f_* : \{0,1,\dots, b-1\} \to \mathbb{Z}^{\ge 0}$ satisfies $f_*(0) = 0$ and $f_*(1) = 1$. It has been proven for $b=10$ and $f_*(m) = m^2$, and various generalizations thereof, that there are arbitrarily long sequences of consecutive positive integers that end up at $1$ under repeated application of $f$. In this paper, we significantly generalize these results, providing a complete classification of digit maps for which, given any periodic point $n$, there are arbitrarily long sequences of consecutive positive integers that end up $n$.
\end{abstract}

\maketitle

\section{Introduction}

In this paper, we look at functions that take in a positive integer and output the sum of its values on the digits of that integer. Precisely, for a fixed base $b$, we start with a function $f_*: \{0,1,\dots,b-1\} \to \Z^{\ge 0}$ and then obtain a map $f: \Z^{\ge 0} \to \Z^{\ge 0}$ given by $f(\sum_{i=0}^n a_ib^i) = \sum_{i=0}^n f_*(a_i)$, where $0 \le a_i \le b-1$. We study long-term iterates of the map $f$; that is, we start with a positive integer $n$ and repeatedly apply $f$, to obtain the sequence $n,f(n),f(f(n)),f(f(f(n)))\dots$. 

\vs

In Richard Guy's book ``Unsolved Problems in Number Theory",  Guy poses many questions regarding $(2,10)$-happy numbers [3]. An \textit{$(e,b)$-happy number} is a number that, under iterates of the digit map $f$ induced by $f_*(m) = m^e$ in base $b$, eventually reaches $1$.  In [4], Pan proved that there exist arbitrarily long sequences of consecutive $(e,b)$-happy numbers assuming that if a prime $p$ divides $b-1$, then the integer $p-1$ does not divide $e-1$. 

\vs

A question appearing in Guy's book [3] is that of gaps in the happy number sequence. It is easy to see that, for any digit map, every positive integer eventually ends up in some finite cycle, i.e. a collection of positive integers $\{n_1,\dots,n_k\}$ such that $f(n_i) = n_{i+1}$ for $1 \le i \le k-1$ and $f(n_k) = f(n_1)$. For example, the cycles generated by the $(2,10)$-happy number digit map are $\{1\}$ and $\{4,16,37,58,89,145,42,20\}$. A gap in the happy number sequence, therefore, corresponds to consecutive numbers that end up in the latter cycle. In this paper, a special case of what we prove is that indeed for any $u$ in an $(e,b)$-happy number cycle, we can find arbitrarily long sequences of consecutive integers that end up in the same cycle as $u$. This answers the question of Guy.

\vs

Significantly more broadly, we provide a complete classification of digit maps for which there are arbitrarily long sequences of consecutive integers ending up in any prespecified cycle. To state our main theorem, we say that a digit map $f$ with base $b$ has a \textit{modular obstruction} if $\gcd\left(f_*(1)-1,\dots,f_*(b-1)-(b-1),b-1\right) > 1$. We call a positive integer $u$ in some cycle a \textit{cycle number} and any positive integer ending up in that cycle a $u$-\textit{integer}.

\vs
 
\begin{theorem}\label{main} Let $f$ be a digit map. If $f$ has a modular obstruction, then for any cycle number $u$, there do not exist two consecutive $u$-integers. If $f$ does not have a modular obstruction, then for any cycle number $u$ and any positive integer $n$, there exist $n$ consecutive $u$-integers. 
\end{theorem}

\vs

For example, working in base $10$, if we construct the digit map $f_*: \{0,\dots,9\} \to \Z^{\ge 0}$ by setting $f_*(0)=0,f_*(1) = 1,$ and $f_*(9) = 7$, and choosing \textit{any} values for $2,3,4,5,6,7$, and $8$, we are guaranteed that there will exist arbitrarily long sequences of consecutive positive integers that end up at $1$ under repeated application of $f$. The result of Theorem \ref{main} consumes the work of H. Pan [4], H. Grundman and E. A. Teeple in [2], and E. El-Sedy and S. Siksek in [1].

\vs

\section{Proof of Theorem \ref{main}}

We first quickly prove the first part of Theorem \ref{main}. Suppose that $f$ has a modular obstruction: there is some $g > 1$ with $g \mid b-1$ and $f_*(m) \equiv m \pmod{g}$ for each $1 \le m \le b-1$. Then, for any $n \in \N$, it holds that $f(n) \equiv n \pmod{g}$; indeed, if $n = \sum_{j=0}^k a_jb^j$, then, since $b \equiv 1 \pmod{g}$, $$f(n) \equiv \sum_{j=0}^k f_*(a_j) \equiv \sum_{j=0}^k a_j \equiv \sum_{j=0}^k a_j b^j \pmod{g}.$$ Therefore, for any $n \in \mathbb{N}$ and $r \ge 1$, the iterate $f^r(n)$ is congruent to $n$ mod $g$. Consequently, if there were a cycle number $u$ and corresponding $n,r_1,r_2 \ge 1$ with $f^{r_1}(n) = f^{r_2}(n+1) = u$, we'd have $n \equiv n+1 \pmod{g}$, absurd. 

\vs

We now move on to the second part of Theorem \ref{main}. We first use a few short results of Pan and introduce new techniques and results in Lemma 3 and Corollary 2.2. Specifically, the proofs of Lemma 1, Corollary 2.1, and Lemma 2 are basically identical to the proofs given by Pan; we just fit them to our notation.

\vs

\begin{lemma} Let $x$ and $m$ be arbitrary positive integers. Then for each $r \ge 1$, there exists a positive integer $l$ such that $$f^r(l+y) = f^r(l)+f^r(y) = x+f^r(y)$$ for each $1 \le y \le m$.
\end{lemma}

\begin{proof}
We use induction on $r$. When $r = 1$, choose a positive integer $s$ such that $b^s > m$ and let $$l_1 = \sum_{j=0}^{x-1}b^{s+j}.$$ Clearly for any $1 \le y \le m$, $$f(l_1+y) = f(l_1)+f(y) = x+f(y).$$ Now assume $r > 1$ and the assertion of Lemma 1 holds for the smaller values of $r$. Note there exists an $m'$ such that $f(y) \le m'$ for $1 \le y \le m$. Therefore, by induction hypothesis, there exists an $l_{r-1}$ such that $$f^{r-1}(l_{r-1}+f(y)) = f^{r-1}(l_{r-1})+f^{r-1}(f(y)) = x+f^r(y)$$ for $1 \le y \le m$. Let $$l_r = \sum_{j=0}^{l_{r-1}-1} b^{s+j}$$ where $s$ satisfies $b^s > m$. Then, $$f^r(l_r) = f^{r-1}(f(l_r)) = f^{r-1}(l_{r-1}) = x$$ and for each $1 \le y \le m$,

\begin{align*}
f^r(l_r+y) &= f^{r-1}(f(l_r+y)) = f^{r-1}(f(l_r)+f(y)) \\
&= f^{r-1}(l_{r-1}+f(y)) = f^{r-1}(l_{r-1})+f^r(y) = f^r(l_r)+f^r(y).
\end{align*}
\end{proof}

\begin{definition} Let $D = D(f_*,b)$ be the set of all positive integers that are in some cycle, that is $u \in D$ if and only if $f^r(u) = u$ for some $r \ge 1$. It is easy to see that $D$ is finite. 
\end{definition}

\begin{definition} Take some $u \in D$. We say a positive integer $n$ is a \textit{$u$-integer} if $f^r(n) = u$ for some $r \ge 1$. We say two positive integers $m,n$ are \textit{concurrently $u$-integers} if for some $r \ge 1$, $f^r(m) = f^r(n) = u$. 
\end{definition}

Note that two $u$-integers $m,n$ are not concurrently $u$-integers only if $u$ belongs to a cycle of length greater than $1$ in $D$ and $m,n$ are at different places in the cycle at a certain time. Note ``concurrently $u$-integers'' is a transitive relation. Now fix $u$ and we will prove that there are arbitrarily long sequences of consecutive $u$-integers. First, we make a reduction.

\vs

\begin{corollary} Assume that there exists $h \in \N$ such that $h+x$ is a $u$-integer for all $x \in D$. Then for arbitrary $m \in \N$, there exists $l \in \N$ such that $l+1,l+2,\dots, l+m$ are $u$-integers. 
\end{corollary}

\begin{proof}
By the definition of $D$, there exists $r \in \N$ such that $f^r(y) \in D$ for all $1 \le y \le m$. By Lemma 1, there exists $l \in \N$ so that $$f^r(l+y) = h+f^r(y)$$ for $1 \le y \le m$. Since $f^r(l+y)$ is then a $u$-integer, $l+y$ is as well, for $1 \le y \le m$.
\end{proof}

\vs

\begin{lemma} Assume that for each $x \in D$ there exists $h_x \in \N$ such that $h_x+u$ and $h_x+x$ are concurrently $u$-integers. Then there exists $h \in \N$ such that $h+x$ is a $u$-integer for each $x \in D$.
\end{lemma}

\begin{proof}
We shall prove that, under the assumption of Lemma 2, for each subset $X$ of $D$ containing $u$, there exists $h_X \in \N$ such that $h_X+x$ is a $u$-integer for each $x \in X$.

\vs

The cases $|X| = 1$ and $|X| = 2$ are clear. Assume $|X| > 2$ and that the assertion holds for every smaller value of $|X|$. Take some $x \in X$, with $x \not = u$. Then $h_x+u$ and $h_x+x$ are concurrently $u$-integers, so take $r \in \N$ large enough so that $f^r(h_x+u) = f^r(h_x+x) = u$ and $f^r(h_x+y) \in D$ for all $y \in X$. Let $X^* = \{f^r(h_x+y) | y \in X\}$. Then, $X^*$ is clearly a subset of $D$ containing $u$ with $|X^*| < |X|$. Therefore, by induction, there exists $h_{X^*} \in \N$ such that $h_{X^*}+f^r(h_x+y)$ is a $u$-integer for each $y \in X$. By Lemma 1, there exists $l \in \N$ satisfying $$f^r(l+h_x+y) = h_{X^*}+f^r(h_x+y)$$ for every $y \in X$. Thus, $h_X := l+h_x$ works. The induction is complete.
\end{proof}

\vs

We now proceed to prove the hypothesis of Lemma 2. Note it suffices to show that for any fixed difference $d$, we can find two concurrent $u$-integers with difference $d$. This is the statement of Corollary 2.2. We first need one more lemma.

\vs

\begin{lemma} Let $h$ be a $u$-integer. Then for every integer $a$, there exists a $u$-integer $l$ such that $l \equiv a \pmod{f(b-1)}$, and such that $l$ and $h$ are concurrently $u$-integers.
\end{lemma}

\begin{proof}
Let $l_1$ be a $u$-integer such that $$l_1 > f(a)+(b-1)f(b-1)\max_{1 \le m \le b-1} f(m).$$ We now find some $l_2$ with $f(l_2) = l_1$ and $l_2 \equiv a \pmod{f(b-1)}$. Since $$\gcd(f(1)-1,\dots,f(b-1)-(b-1),f(b-1)) = 1,$$ we may take $r_1,\dots,r_{b-1} \in \{0,\dots,f(b-1)\}$ so that $$r_1(1-f(1))+\dots+r_{b-1}(b-1-f(b-1)) \equiv f(a)-l_1 \pmod{f(b-1)}.$$ Note that $$L := l_1-f(a)-r_1f(1)-\dots-r_{b-1}f(b-1)$$ satisfies $L \ge 1$. By the pigeonhole principle, there is some $b' \in \{0,\dots,b-1\}$ such that $b^j \equiv b' \pmod{f(b-1)}$ for infinitely many $j$. Let $j_1 < j_2 < \dots < j_L < t_1^{(1)} < \dots < t_{r_1}^{(1)} < \dots < t_1^{(b-1)} < \dots < t_{r_{b-1}}^{(b-1)}$ satisfy $b^{j_i} \equiv b^{t_s^{(k)}} \equiv b' \pmod{f(b-1)}$ for each $i,s$, and $k$, and satisfy $b^{j_1} > a$. 

Let $$l_2 = a+\sum_{n=1}^L b^{j_n}+\sum_{m=1}^{b-1}\sum_{j=1}^{r_m} mb^{t_j^{(m)}}.$$ Due to the inequality $b^{j_1} > a$, we have $$f(l_2) = f(a)+L+r_1f(1)+\dots+r_{b-1}f(b-1) = l_1,$$ and due to the choice of $r_i$'s, we have $$l_2 \equiv a+b'\left[L+r_1+2r_2+\dots+(b-1)r_{b-1}\right] \equiv a \pmod{f(b-1)}.$$

Now we generate $l_3,l_4,\dots$ inductively by choosing $l_{n+1}$ so that $l_{n+1} \equiv a \pmod{f(b-1)}$ and $f(l_{n+1}) = l_n$. Note that since the cycle that $u$ is in is finite, it must be that one of the $l_n$'s is concurrently a $u$-integer with $h$. \end{proof}

\vs

\begin{corollary} For each $x \in \N$, there is a $u$-integer $l$ such that $l$ and $l+x$ are concurrently $u$-integers.
\end{corollary}

\begin{proof} Fix $x \in \N$. Take $s \in \N$ such that $b^s > x$. Let $x_1 = b^s-x$. Take a $u$-integer $h'$ such that $$h' \equiv f(x_1) \pmod{f(b-1)}.$$ Let $V$ be the cycle set that $u$ is in. By Lemma 3, for each $v' \in V$, there exists $l_{v'}$ such that $l_{v'} \equiv 1 \pmod{f(b-1)}$, and $l_{v'}$ and $v'$ are concurrently $u$-integers. Fixing an $l_{v'}$ for each $v' \in V$, let $M = \max_{v' \in V} l_{v'}$.

\vs

Since the proof of Lemma 3 guarantees infinitely many $u$-integers in a given residue, we may (and do) fix $h > f(x_1)+M$ to be a $u$-integer with $h \equiv f(x_1) \pmod{f(b-1)}$. Let $v$ be in the cycle of $u$ so that $h$ and $v$ are concurrently $u$-integers. Now take the $u$-integer $N = l_{v}$ so that $N \equiv 1\pmod{f(b-1)}$, and $N$ and $v$ are concurrently $u$-integers. Take a positive integer $t$ so that $b^t > b^{s+\lfloor \frac{h}{f(b-1)}\rfloor+1}$. Let $x_2 = x_1+b^t\sum_{j=1}^{N-1} b^j$. Note $f(x_2) = f(x_1)+(N-1)$ since $b^t > b^s > x_1$. Thus, $$f(x_2) \equiv f(x_1) \equiv h \pmod{f(b-1)}.$$ Also, $f(x_2) = f(x_1)+(N-1) \le f(x_1)+M-1 < h$. Write $h = f(b-1)k+f(x_2)$ and note that we have $k > 0$. Also note $k \le \lfloor \frac{h}{f(b-1)} \rfloor + 1 < t-s$. Let $$l = x_2+\sum_{j=0}^{k-1}(b-1)b^{s+j}.$$

Then, 

\begin{align*}
f(l) &= f\left(x_1+b^t\sum_{j=1}^{N-1} b^j + b^s\sum_{j=0}^{k-1}(b-1)b^j\right) \\
&= f\left(x_1+b^s[b^{t-s}\sum_{j=1}^{N-1} b^j + \sum_{j=0}^{k-1} (b-1)b^j]\right) \\
&= f(x_1)+f(b^{t-s}\sum_{j=1}^{N-1} b^j + \sum_{j=0}^{k-1} (b-1)b^j),
\end{align*}

\noindent and since $\sum_{j=0}^{k-1}(b-1)b^j = b^k-1 < b^{t-s}$, we have $$f(l) = f(x_1)+(N-1)+kf(b-1) = f(x_2)+kf(b-1) = h.$$ Further, $$f(l+x) = f\left(b^s+\sum_{j=0}^{k-1}(b-1)b^{s+j} + b^t\sum_{j=1}^{N-1} b^j\right) = f\left(b^{s+k}+b^t\sum_{j=1}^{N-1} b^j\right),$$ which is equal to $N$. Since $h$ and $N$ are concurrently $u$-integers, it follows that $l$ and $l+x$ are concurrently $u$-integers, as desired.
\end{proof}

\vs

Theorem 1.1 now follows from Corollary 2.1, Lemma 2, and Corollary 2.2.

\vs

\section{Acknowledgments}

I would like to thank Serin Hong for reading this carefully and for providing great feedback. I also thank Akash Parikh, Omer Tamuz, and an anonymous referee for providing insightful comments.

\end{document}